\documentclass[12pt]{article}         

 \usepackage[T1]{fontenc}              
 \usepackage[latin1]{inputenc}         
      \input{dehyphtex}                
 \usepackage[width=16cm, height=22cm]{geometry}   
 \usepackage{mathtools} 
 \usepackage{amsmath,amssymb}          
\usepackage{hyperref} 
 \usepackage[amsmath,thmmarks,hyperref]{ntheorem}

 \usepackage{icomma}                   
 \usepackage{enumerate}                
 \usepackage{color}                    
 \usepackage{setspace}                 
 \usepackage{needspace}                
 \usepackage{url}                      
 \usepackage{mathrsfs}                 
 \usepackage{graphicx}                 
 \usepackage{slashed}
 \usepackage{esint}

 \usepackage{xypic}

 \numberwithin{equation}{section}
 
 \usepackage[numbers,square]{natbib}
\theoremstyle{nonumberplain}  
\theoremheaderfont{\itshape}  
\theorembodyfont{\normalfont}  
\theoremseparator{.}  
\theoremsymbol{$\Box$}
\newtheorem{proof}{Proof} 
  
\theoremstyle{plain}  
\theoremheaderfont{\normalfont\bfseries}  
\theorembodyfont{\itshape}  
\theoremsymbol{~}
\theoremseparator{.}  
\newtheorem{proposition}{Proposition}[section]  
\newtheorem{corollary}[proposition]{Corollary}  
\newtheorem{lemma}[proposition]{Lemma}  
\newtheorem{theorem}[proposition]{Theorem}   

\theorembodyfont{\normalfont}  
 
\newtheorem{remark}[proposition]{Remark}
\newtheorem{example}[proposition]{Example}

\newtheorem{definition}[proposition]{Definition}

\theoremstyle{nonumberplain}
\theoremheaderfont{\normalfont\bfseries}  
\theorembodyfont{\itshape}  
\theoremsymbol{~}
\theoremseparator{.}

\newcommand{\R}{\mathbb{R}}

\newcommand{\V}{\mathcal{V}}
\newcommand{\N}{\mathbb{N}}

\newcommand{\C}{\mathbb{C}}
\newcommand{\dd}{\mathrm{d}}

\newcommand{\id}{\mathrm{id}}

\newcommand{\<}{\left\langle}
\renewcommand{\>}{\right\rangle}

\renewcommand{\hat}{\widehat}

\title{Heat Kernels as Path Integrals}
\author{Matthias Ludewig}

\begin{document}

\maketitle

\begin{center}
  The University of Adelaide\\
  Adelaide SA 5005 \\
  Australia \\
 matthias.ludewig@adelaide.edu.au
 
 \vspace{1cm}
\begin{small}
{\em This article is an exanded version of a survey article originally published in} \cite{SFBReport}.
\end{small}

\end{center}

\section{Introduction}

Path integrals are an ubiquitous concept in quantum mechanics and quantum field theory. Roughly speaking, the idea is that for a classical observable $\mathcal{O}$, the expectation value $\langle\hat{\mathcal{O}}\rangle$ of the corresponding "quantized" observable $\hat{\mathcal{O}}$ can be calculated with a path integral of the form
\begin{equation} \label{PrototypicalPathIntegral}
  \langle \hat{\mathcal{O}}\rangle~ \propto~ \int_{\mathcal{F}} \exp \bigl(i S[\gamma] \bigr) \mathcal{O}[\gamma] \mathcal{D}\gamma,
\end{equation}
where the integration domain $\mathcal{F}$ is a certain space of "fields" (usually some space of maps $\gamma: \Sigma \rightarrow M$) and 
\begin{equation*}
  S[\gamma] = \int_\Sigma \ell\bigl (\gamma(x), d\gamma(x) \bigr) \dd x
\end{equation*}
is some classical action, given be the integral over a Lagrangian. In \eqref{PrototypicalPathIntegral}, the integral $\mathcal{D}\gamma$ denotes some kind of Lebesgue type measure on the space of fields $\mathcal{F}$. This is very problematic mathematically, because in most situations of interest, no such measure $\mathcal{D}\gamma$ exists. 

An example for this situation is the following.  Let $K_t(x, y)$ be the transition density of the Schr\"odinger equation
\begin{equation} \label{SchroedingerEquation}
  -i \frac{\partial}{\partial t} u(t, x) + \frac{1}{2}\Delta u(t, x) + V(x) u(t, x) = 0, ~~~~~~~~~ u(0, x) = u_0(x),
\end{equation}
on some (compact) Riemannian manifold $M$, characterized by the property that the solutions $u(t, x)$ are given in terms of the initial value $u_0$ by the integral
\begin{equation} \label{ConvolutionSolution}
  u(t, x) = \int_M K_t(x, y) u_0(y) \dd y.
\end{equation}
$K_t(x, y)$ can be formally represented by the path integral of the form \eqref{PrototypicalPathIntegral}, where the space of fields $\mathcal{F}$ is the space of continuous paths that travel from $x$ to $y$ in time $t$. Concretely, 
\begin{equation} \label{SchroedingerKernelPathIntegral}
  K_t(x, y) \stackrel{\text{formally}}{=} \fint_{\mathcal{F}}  \exp \bigl( i S[\gamma] \bigr) \mathcal{D}\gamma
\end{equation}
where  $S[\gamma]$ is the classical action
\begin{equation} \label{ActionWithPotential}
  S[\gamma] = \frac{1}{2} \int_0^t \bigl|\dot{\gamma}(s)\bigr|^2 + \int_0^t V\bigl(\gamma(s)\bigr) \dd s.
\end{equation}
and the slash over the integral sign in \eqref{SchroedingerKernelPathIntegral} denotes a suitable normalization.

Making this rigorous is still not easy due to the oscillatory exponent in \eqref{SchroedingerKernelPathIntegral} and the corresponding singular nature of the Schr\"odinger kernel $K_t(x, y)$ on general Riemannian manifolds. However, if we instead consider the heat equation
\begin{equation} \label{HeatEquationWithPotential}
  \frac{\partial}{\partial t} u(t, x) + \frac{1}{2}\Delta u(t, x) + V(x) u(t, x) = 0, ~~~~~~~~~ u(0, x) = u_0(x),
\end{equation}
we obtain the analogous formula
\begin{equation}  \label{HeatKernelPathIntegral}
  K_t(x, y) \stackrel{\text{formally}}{=} \fint_{\mathcal{F}}  \exp \bigl(- S[\gamma] \bigr) \mathcal{D}\gamma.
\end{equation}
for the heat kernel $K_t(x, y)$. Here the exponent is exponentially decaying, which simplifies the analysis considerably. \eqref{HeatKernelPathIntegral} can be made rigorous by approximating the space of paths $\mathcal{F}$ by spaces of piecewise geodesics. In this note, we give an overview over the results obtained in this direction.

To introduce the relevant ideas, we start out with a review of path integral formulas in $\R^n$ and discuss the relation to the Wiener measure. Afterwards, we give path integral approximation theorems for solutions to the heat equation on compact Riemannian manifolds. Finally, we indicate how the heat kernel can be represented by a path integral and show that the short-time asymptotic expansion of the heat kernel is related to the stationary phase expansion of the corresponding path integral.

\section{Approximation of Path Integrals: the flat Case}

Let us start with a classical calculation which can be found in most physics textbooks on quantum mechanics and which probably originates in \cite{FeynmanHibbs}. Consider the heat equation
\begin{equation} \label{HeatEquation}
  \frac{\partial}{\partial t} u(t, x) + \frac{1}{2}\Delta u(t, x) = 0, ~~~~~~~~~~~ u(0, x) = u_0(x)
\end{equation}
in Euclidean space $\R^n$. In this case, the corresponding heat kernel is given by the explicit formula
\begin{equation*}
  K_t(x, y) = (2\pi t)^{-n/2} \exp \left( -\frac{|x-y|^2}{2 t} \right).
\end{equation*}
The integrable solution to \eqref{HeatEquation} is then given in terms of this fundamental solution by formula \eqref{ConvolutionSolution}. If now $\tau = \{ 0 = \tau_0 < \tau_1 < \dots < \tau_N = t\}$ is a partition of the interval $[0, t]$ and $\Delta_j\tau := \tau_j - \tau_{j-1}$ denotes the corresponding increments, we can use the Markhov property of the heat kernel to obtain (writing $x_0 := x$)
\begin{equation} \label{Calculation1}
\begin{aligned}
  u(t, x) &= \int_{\R^n} \cdots \int_{\R^n} \prod_{j=1}^N K_{\Delta_j\tau}(x_j, x_{j-1}) \,u_0(x_N)\, \dd x_1 \cdots \dd x_N\\
  &= \prod_{j=1}^N (2 \pi \Delta_j\tau)^{-n/2} \int_{\R^{n\times N}} \exp\left(-\frac{1}{2}\sum_{j=1}^N \frac{|x_j-x_{j-1}|^2}{\Delta_j\tau} \right) u_0(x_N) \dd x.
\end{aligned}
\end{equation}
Let $\gamma: [0, t] \rightarrow \R^n$ be the piecewise linear path with $\gamma(\tau_j) = x_j$ so that $\gamma$ is given by the explicit formula
\begin{equation*}
  \gamma(\tau_{j-1}+s) = x_{j-1} + \frac{s}{\Delta_j\tau}(x_j - x_{j-1}), ~~~~~~~ s \in [0, \Delta_j\tau].
\end{equation*}
for $j=1, \dots, N$. Then for its action, one finds
\begin{equation*}
  S_0[\gamma] := \frac{1}{2} \int_0^1 \bigl|\dot{\gamma}(s)\bigr|^2 \dd s = \frac{1}{2}\sum_{j=1}^N \int_{\tau_{j-1}}^{\tau_j} \left|\frac{x_j-x_{j-1}}{\Delta_j\tau}\right|^2 \dd s
 = \frac{1}{2}\sum_{j=1}^N \frac{|x_j-x_{j-1}|^2}{\Delta_j\tau}.
\end{equation*}
Hence the exponent in \eqref{Calculation1} is exactly the action of the piece-wise linear path with the nodes $x_0, \dots, x_N$, subordinated to the partition $\tau$. Therefore, the integral over $\R^{n\times N}$ in \eqref{Calculation1} can be replaced by an integral over the space 
\begin{equation*}
  H_{x;\tau}(\R^n) := \bigl\{ \gamma \in C^0([0, t], \R^n) \mid \gamma(0) = x, \gamma|_{[\tau_{j-1}, \tau_j]}~\text{is a straight line for}~j=1, \dots, N\bigr\}.
\end{equation*}
In this flat setup, $H_{x;\tau}(\R^n)$ is a an affine space. Its tangent spaces are all canonically isomorphic to the space of piece-wise linear vector fields $X$ (subordinated to the partition $\tau$) with $X(0) = 0$. If we put on it the (flat) Riemannian metric
\begin{equation*}
(X, Y)_{H^1} = \int_0^t \< X^\prime(s), Y^\prime(s)\> \dd s,
\end{equation*}
which is non-degnerate since $X$ and $Y$ have vanishing initial values, then the $\tau$-dependent pre-factor in \eqref{Calculation1} is absorbed into the volume measure corresponding to the metric and we obtain that
\begin{equation} \label{ExactPathIntegral}
  u(t, x) = (2 \pi)^{-nN/2} \int_{H_{x;\tau}(\R^n)} \exp\bigl({-S_0[\gamma]}\bigr) \,u_0 \bigl(\gamma(t)\bigr) \dd^{H^1}\gamma.
\end{equation}
for each partition $\tau = \{0 = \tau_0 < \tau_1 < \dots < \tau_N = t\}$ of the interval $[0, t]$. Notice that $\dim(H_{x;\tau}(\R^n)) = nN$ so that the pre-factor of \eqref{ExactPathIntegral} only depends on the dimension of the integration domain. This gives our first "path integral formula": The solution $u(t, x)$ to the heat equation is written as an integral over a certain (finite-dimensional) space of paths. 

\begin{remark}The heat equation also describes particle diffusion, and it is elucidating to give an interpretation of formula \eqref{ExactPathIntegral} in terms of statistical mechanics: For a particle to move from a point $y$ to a point $x$ in time $t$, it has to take some continuous path $\gamma: [0, t] \rightarrow \R^n$ with $\gamma(0) = x$ and $\gamma(t) = y$. If $u_0$ is interpreted as an initial particle density, then \eqref{ExactPathIntegral} tells us that the new density after time $t$ is given at the point $x$ by averaging over all paths that end at this point, where we have to weight with the initial distribution and the classical action: It is unlikely for a particle to take a path with a large action.\end{remark}

In formula \eqref{ExactPathIntegral}, one integrates only over some space of polygon paths rather than over the space of {\em all} continuous paths $C([0, t], \R^n)$. On the latter, there is for each point $x \in \R^n$ an important probability measure, the {\em Wiener measure} $\mathbb{W}_x$. It is characterized by the property that for all bounded and continuous cylinder functions $f$ on $C([0, t], \R^n)$ (i.e.\ functions for which there exists a continuous function $F$ on $\R^{n\times N}$ and some partition $\tau$ of the interval $[0, t]$ such that $f(\gamma) = F(\gamma(\tau_1), \dots, \gamma(\tau_N))$), one has
\begin{equation} \label{WienerCharacterization}
  \int_{C([0, t], \R^n)} f(\gamma) \,\dd \mathbb{W}_x(\gamma) = \int_{\R^n} \cdots \int_{\R^n} F(x_1, \dots, x_N) \prod_{j=1}^N K_{\Delta_j\tau}(x_j, x_{j-1}) \dd x_1\cdots \dd x_N
 \end{equation}
By Kolmogorov's extension theorem, this property characterizes the Wiener measure uniquely \cite[Thm.~2.1.5]{oeksendal}. The Wiener measures are related to our path integrals \eqref{ExactPathIntegral} as follows. For a partition $\tau$ of the interval $[0, t]$, let
    \begin{equation*}
    \mathbb{W}_{x, \tau} :=  (j_{x,\tau})_* \Bigl[(2 \pi t)^{-nN/2} \exp\bigl({-S_0[\gamma]}\bigr) \dd^{H^1}\gamma\Bigr]
  \end{equation*}
 be the measure on $C([0, t], \R^n)$ given by pushforward of the measure on $H_{x;\tau}(M)$ by the inclusion map $j_{x,\tau}: H_{x;\tau}(M) \rightarrow C([0, t], \R^n)$. Then we have the following result.
 
\begin{proposition}[Approximation of Wiener Measure] \label{PropWienerConvergence}
 As the mesh $|\tau| = \max \Delta_j \tau$ of the partitions tends to zero, $\mathbb{W}_{x, \tau}$ converges against the Wiener measure $\mathbb{W}_x$ in the weak-$*$-topology of finite Borel measures on $C([0, t], \R^n)$.
\end{proposition}

\begin{proof}
  By \eqref{WienerCharacterization}, we have
  \begin{equation*}
    \int_{C([0, t], \R^n)} f(\gamma) \dd \mathbb{W}_{x, \tau}(\gamma) = \int_{C([0, t], \R^n)} f(\gamma^\tau) \dd \mathbb{W}_x(\gamma),
  \end{equation*}
  where for $\gamma \in C^0([0, t], \R^n)$, we denote by $\gamma_{x,\tau}$ the path in $H_{x;\tau}(M)$ with $\gamma_{x,\tau}(\tau_j) = \gamma(\tau_j)$ for all $j=1, \dots, N$. 
  Using Lebesgue's theorem of dominated convergence, the proposition therefore follows from the subsequent lemma.
\end{proof}

\begin{lemma} \label{LemmaPointwiseConvergence}
  For a continuous function $f$ on $C([0, t], \R^n)$, denote by $f_{x, \tau}$ the continuous function given by $f_{x,\tau}(\gamma) = f(\gamma_{x,\tau})$. Then we have $f_{x,\tau} \rightarrow f$ pointwise $\mathbb{W}_x$-almost surely, as $|\tau| \rightarrow 0$.
\end{lemma}

\begin{proof}
Since paths that start at $x$ have full measure with respect to $\mathbb{W}_x$, the lemma follows if we show that $\|\gamma_{x,\tau} -\gamma\|_\infty \rightarrow 0$ as $|\tau| \rightarrow 0$ for all $\gamma$ with $\gamma(0) = x$, since then also $f(\gamma_{x,\tau}) \rightarrow f(\gamma)$, by continuity of $f$. For $s \in [\tau_{j-1}, \tau_j]$, we have
\begin{equation*}
  \bigl|\gamma(s) - \gamma_{x,\tau}(s)\bigr| \leq \bigl|\gamma(s) - \gamma(\tau_{j-1})\bigr| + \bigl|\gamma_{x,\tau}(s) - \gamma(\tau_{j-1})\bigr| \leq \bigl|\gamma(s) - \gamma(\tau_{j-1})\bigr| + \bigl|\gamma(\tau_j) - \gamma(\tau_{j-1})\bigr|,
\end{equation*}
where we used that $\gamma_{x,\tau}$ is piecewise linear and that the paths $\gamma$, $\gamma_{x,\tau}$ coincide at the nodes $\tau_0, \dots, \tau_N$. As a function defined on a compact interval, $\gamma$ is uniformly continuous, hence for each $\varepsilon>0$, there exists $\delta>0$ such that $|\gamma(s_1)-\gamma(s_2)| < \varepsilon/2$ whenever $|s_1-s_2|< \delta$. If now $|\tau|<\delta$, then $|\gamma(s) - \gamma_{x,\tau}(s)| < \varepsilon$ for each $s \in [0, t]$.
\end{proof}

Lemma.~\ref{LemmaPointwiseConvergence} also allows us to obtain the following non-trivial path integral representation of solutions to the heat equation with potential \eqref{HeatEquationWithPotential}. Namely, one has the following

\begin{corollary}[Heat Equation with Potential] \label{CorollaryPotential}
  Let $u(t, x)$ be a bounded solution of the heat equation with potential \eqref{HeatEquationWithPotential}. Then
  \begin{equation} \label{PathIntegralPotential}
    u(t, x) = \lim_{|\tau|\rightarrow 0} \fint_{H_{x;\tau}(M)} \exp \bigl(-S[\gamma]\bigr) u_0\bigl(\gamma(t)\bigr) \dd^{H^1} \gamma, 
  \end{equation}
  where $S$ denotes the action functional \eqref{ActionWithPotential} and the slash over the integral sign denotes a normalization of the path integral through division by $(2\pi )^{\dim(H_{x;\tau}(M))/2}$.
\end{corollary}

\begin{proof}
  By \eqref{WienerCharacterization},
  \begin{equation*}
  \fint_{H_{x;\tau}(\R^n)} \exp \bigl(-S[\gamma]\bigr) u_0\bigl(\gamma(t)\bigr) \dd \gamma = \int_{C([0, t], \R^n)} \!\!\!\!\!\exp \left(-\int_0^t V\bigl(\gamma_{x,\tau}(s)\bigr)\dd s\right) u_0\bigl(\gamma_{x,\tau}(t)\bigr) \dd \mathbb{W}(\gamma)
  \end{equation*} 
  The result now follows from Lemma~\ref{LemmaPointwiseConvergence} and the Feynman-Kac formula, which states that
  \begin{equation} \label{FeynmanKac}
    u(t, x) = \int_{C([0, t], \R^n)} \!\!\!\!\!\exp \left(-\int_0^t V\bigl(\gamma(s)\bigr)\dd s\right) u_0\bigl(\gamma(t)\bigr) \dd \mathbb{W}_x(\gamma),
  \end{equation}
  see \cite{oeksendal}, Thm.~8.2.1.
\end{proof}

Of course, for non-zero potential, one cannot expect $u(t, x)$ to be given {\em exactly} by an integral over $H_{x;\tau}(\R^n)$ as in \eqref{ExactPathIntegral} and the mesh-limit is necessary, as in \eqref{PathIntegralPotential}. 

\section{Approximation of Path Integrals: The curved Case}

A natural question from a geometric point of view is now whether results like Corollary~\ref{CorollaryPotential} hold in a curved setup, i.e.\ when $\R^n$ is replaced by a smooth manifold. Of course, it is well-known that the Feynman-Kac formula \eqref{FeynmanKac} is true in this case (see \cite{bpFeynmanKac}). On the finite-dimensional approximation side, one first has to replace the spaces $H_{x;\tau}(\R^n)$ by their curved analogs. For a compact Riemannian manifold $M$ and a partition $\tau=\{0 = \tau_0 < \tau_1 < \dots < \tau_N = t\}$, we define
\begin{equation*}
  H_{x;\tau}(M) := \bigl\{ \gamma \in C([0, t], M) \mid \gamma(0) = x, \gamma|_{(\tau_{j-1}, \tau_j)}~\text{is a geodesic for}~j=1, \dots, N\bigr\}.
\end{equation*}
This is then a smooth manifold of dimension $nN$, where the tangent space at a path $\gamma$ consists of the continuous vector fields $X$ along $\gamma$ with $X(0) = 0$ and such that $X|_{(\tau_{j-1}, \tau_j)}$ is a Jacobi field. If $M$ is flat, the two metrics
\begin{equation} \label{ContinuousH1}
  (X, Y)_{H^1} := \int_0^t \bigl\langle \nabla_s X(s), \nabla_s Y(s)\bigr\rangle \dd s
\end{equation}
and its discretized analog
\begin{equation} \label{DiscreteH1}
  (X, Y)_{\Sigma\text{-}H^1} := \sum_{j=1}^N \bigl\langle \nabla_s X(\tau_{j-1}+), \nabla_sY(\tau_{j-1}+) \bigr\rangle \Delta_j\tau
\end{equation}
coincide, but in the general case, they differ. It turns out that the discretized metric is the correct choice for the finite-dimensional approximations: When endowed with this metric, the curved analog of \eqref{PathIntegralPotential} holds, as was proved in \cite{AnderssonDriver} and, in more general setups,  in \cite{bpapproximation} and \cite{LudewigBoundary}. 

More generally, an analog of Corollary~\ref{CorollaryPotential} is true for Laplace-type operators, acting on sections of a vector bundle. If $\V$ is a vector bundle over a compact Riemannian manifold $M$, which is equipped with a fiber metric, we consider differential operators of the form
\begin{equation} \label{LaplaceConnectionPotential}
H = \frac{1}{2}\nabla^*\nabla +V,
\end{equation}
where $\nabla$ is a metric connection on $\V$ and $V$ is a smooth and symmetric endomorphism field of $\V$ (this decomposition is unique).
The corresponding heat equation is then
\begin{equation} \label{GeneralHeatEquation}
 \frac{\partial}{\partial t} u(t, x) + Hu(t, x) = 0, ~~~~~~~~ u(0, x) = u_0(x)
\end{equation}
for sections $u, u_0$ of the bundle $\V$. Such an operator determines a {\em path-ordered exponential} $\mathcal{P}(\gamma) = \mathcal{P}^{\nabla, V}(\gamma)$ along absolutely continuous paths $\gamma: [0, t] \rightarrow M$, defined by $\mathcal{P}(\gamma) = P(t)$, where $P(s) \in \mathrm{Hom}(\V_{\gamma(0)}, \V_{\gamma(s)})$ is the unique solution of the ordinary differential equation
\begin{equation*}
  \nabla_s P(s) = V\bigl(\gamma(s)\bigr) P(s)~~~~~~~~ P(0) = \mathrm{id}.
\end{equation*}
For example, if $H = \frac{1}{2}\nabla^* \nabla$ (i.e.\ the potential is zero), then $\mathcal{P}(\gamma) = [\gamma\|_0^t]$, the parallel transport along $\gamma$ and if $H= \frac{1}{2}\Delta + V$, the Laplace-Beltrami operator with a potential, then the corresponding $\mathcal{P}(\gamma)$ equals the Feynman-Kac integrand in \eqref{FeynmanKac}. The result is now the following.

\begin{theorem}[Vector-valued Heat Equation] \label{ThmPathIntegralManifold}
Let $H$ be a self-adjoint Laplace type operator, acting on sections of a metric vector bundle $\V$ over a compact Riemannian manifold and let $\mathcal{P}(\gamma)$ be the path-ordered exponential determined by $L$ as above. Then the solution of the corresponding heat equation \eqref{GeneralHeatEquation} is given by
\begin{equation*}
  u(t, x) = \lim_{|\tau|\rightarrow 0} \fint_{H_{x;\tau}(M)} \exp\bigl({-S_0(\gamma)}\bigr) \mathcal{P}(\gamma)^{-1} u\bigl(\gamma(t)\bigr) \dd^{\Sigma\text{-}H^1} \gamma,
\end{equation*}
where the limit goes over any sequence of partitions of the interval $[0, t]$, the mesh of which goes to zero, and the slash over the integral sign denotes divison by $(2\pi)^{\mathrm{dim}(H_{x;\tau}(M))/2}$. This result is true for $u_0$ in any of the spaces $C^0(M, \V)$ or $L^p(M, \V)$ with $1 \leq p < \infty$, with the convergence is in the respective space.
\end{theorem}

\begin{example}[Quantizing Hamiltonian Functions] \label{ExampleQuantization}
In the classical mechanics of point particles, one considers Hamiltonian functions on phase space, which are smooth functions on the cotangent bundle of a Riemannian manifold $M$ (we assume it to be compact here). A typical electromagnetic Hamiltonian function is of the form
\begin{equation} \label{StandardHamiltonian}
  h(x, p) = |p - \omega(x)|^2 + V(x), ~~~~~~~ x \in M, ~~p \in T^*_xM
\end{equation}
where $\omega \in \Omega^1(M)$ is a given one-form and $V \in C^\infty(M)$ is a potential. The corresponding quantum mechanical Hamiltonian is the operator $H = \frac{1}{2}\nabla^*\nabla + V$, where $\nabla = d + i\omega$ is the connection determined by $\omega$. The corresponding time evolution operator is $e^{itH}$, which we cannot deal with; however, the "wick-rotated" solution operator $e^{-tH}$ {\em can} be represented by a path integral as follows. It involves the Lagrange function 
\begin{equation*}
  \ell(x, v) = \frac{1}{2}|v|^2 + \omega(x)\cdot v - V(x), ~~~~~~~ x \in M, ~~ v \in T_x M
\end{equation*}
associated to the Hamiltonian function $h$. Namely, the term $\mathcal{P}(\gamma)^{-1}$ can be computed explicitly for our particular Hamiltonian $H$, namely
\begin{equation*}
  \exp\bigl({-S_0(\gamma)}\bigr) \mathcal{P}(\gamma)^{-1}  = \exp \left( \int_0^t \Bigl( -\frac{1}{2} \bigl|\dot{\gamma}(s)\bigr|^2 + i \omega\bigl(\gamma(s)\bigr) \dot{\gamma}(s) - V\bigl(\gamma(s)\bigr) \Bigr)\dd s \right).
\end{equation*}
We obtain the path integral formula
\begin{equation} \label{ConfigurationSpace}
  u(t, x) = \lim_{|\tau|\rightarrow 0} \fint_{H_{x;\tau}(M)} \exp \left( \int_0^t \ell\bigl(\gamma(s), i\dot{\gamma}(s)\bigr)\dd s\right) u\bigl(\gamma(t)\bigr)\dd^{\Sigma\text{-}H^1} \gamma,
\end{equation}
where we extended $\ell$ to a fiber-wise polynomial on $TM \otimes \C$. Notice the imaginary unit in the Lagrangian; it is due to the fact that we substituted $t \mapsto -it$ in order to be able to use our results.

\eqref{ConfigurationSpace} is often called {\em configuration space path integral}, since the paths run in configuration space $M$ rather than in momentum space $T^*M$. We can also represent $u(t, x)$ by a {\em momentum space path integral}, which is an integral over the space $\Omega_{x;\tau}(M)$ of tuples $(\gamma, \xi)$, where $\gamma \in H_{x;\tau}(M)$ and $\xi$ is a (non-continuous) section of $L^2([0, t], \gamma^*T^*M)$ that is parallel on each subinterval $[\tau_{j-1}, \tau_j]$. $\Omega_{x;\tau}(M)$ is naturally a vector bundle over $H_{x;\tau}(M)$ which carries the $L^2$ metric
\begin{equation*}
  (\xi, \eta)_{L^2} := \int_0^t \bigl\langle \xi(s), \eta(s) \bigr\rangle \dd s
\end{equation*}
induced from $L^2([0, t], \gamma^*TM)$. Elementary calculations with Gaussian integrals now show that
\begin{equation*}
  u(t, x) = \lim_{|\tau|\rightarrow 0} \fint_{\Omega_{x;\tau}(M)} \exp \left( \int_0^t \Bigl[ i \xi(s) \cdot \dot{\gamma}(s) - h\bigl(\gamma(s), \xi(s)\bigr) \Bigr]\dd s\right) u\bigl(\gamma(t)\bigr)\dd (\gamma, \xi),
\end{equation*}
where we integrate with respect to the induced metric on the total space $\Omega_{x;\tau}(M)$ and the slash over the integral sign denotes division by $(2 \pi)^{\dim(\Omega_{x;\tau}(M))/2} = (2 \pi)^{\dim(H_{x;\tau}(M))}$. Notice that in this formulation, the Hamiltonian appears instead of the Lagrangian.
\end{example}

\begin{example}[The Dirac Operator on Spinors]
  Let $M$ be a spin manifold and denote by $\slashed{D}$ the Dirac operator, acting on spinors. Then the famous Lichnerowicz formula asserts
  \begin{equation*}
    \slashed{D}^2 = \nabla^*\nabla + \frac{1}{4} \mathrm{scal},
  \end{equation*} 
  where $\nabla$ is the metric connection on spinors induced by the Levi-Civita connection. Hence the solution $\psi(t, x)$ to the heat equation
  \begin{equation*}
   \frac{\partial}{\partial t} \psi(t, x) + \frac{1}{2}\slashed{D}^2\psi(t, x) = 0, ~~~~~~~\psi(0, x) = \psi_0(x)
  \end{equation*}
  on spinors is given by
  \begin{equation*}
    \psi(t, x) = \lim_{|\tau|\rightarrow 0} \fint_{H_{x;\tau}(M)} \exp\left(-S_0[\gamma] - \frac{1}{8} \int_0^t \mathrm{scal}\bigl(\gamma(s)\bigr) \dd s \right) [\gamma\|_0^t]^{-1} \psi_0\bigl(\gamma(t)\bigr)\dd^{\Sigma\text{-}H^1}\gamma,
  \end{equation*}
  where $[\gamma\|_0^t]$ denotes the parallel transport  with respect to the Levi-Civita connection on spinors.
  \end{example}

There are various ways to prove Thm.~\ref{ThmPathIntegralManifold}. The original proof by Andersson and Driver relies on techiques from stochastic analysis: An essential step is the observation that the {\em development map} sending piecewise polygon paths in $T_xM$ to the piecewise geodesic paths in $H_{x;\tau}(M)$ is measure preserving (see \cite{AnderssonDriver}, Def.~2.2 and Thm.~4.10); the result is then proven conceptually similar to Prop.~\ref{PropWienerConvergence} (compare Thm.~4.14 in \cite{AnderssonDriver}).

The proof by B\"ar and Pf\"affle (see \cite{bpapproximation}) essentially uses Chernoff's theorem, which asserts that if $P_t$, $t \in [0, \infty)$ is a uniformly bounded family of linear operators on a Banach spaces with $P_0 = \id$ which satisfies $\|P_t\| = O(t)$ and possesses an infinitesimal generator $-H$ (where the infinitesimal generator is defined just as in the case of an ordinary semigroup of operators), then one has
\begin{equation*}
  \lim_{|\tau|\rightarrow 0} P_{\Delta_1\tau} \cdots P_{\Delta_N\tau} = e^{-tH},
\end{equation*}
where the limit goes over any sequence of partitions of the interval $[0, t]$, the mesh of which tends to zero and holds in the strong operator topology (see \cite[Thm.~2.8]{bpapproximation}, Prop.~1 in \cite{WeizsaeckerSmolyanov05} or \cite{chernoff}). If now $\tau^t := \{0 = \tau_0^t < \tau_1^t = t\}$ denotes the trivial partition of the interval $[0, t]$, one verifies that on the one hand,
\begin{equation} \label{ProperFamily}
  (P_t u)(x) := \fint_{H_{x;\tau^t}(M)} \exp\bigl({-S_0(\gamma)}\bigr) \mathcal{P}(\gamma)^{-1} u\bigl(\gamma(t)\bigr) \dd^{\Sigma\text{-}H^1} \gamma
\end{equation}
satisfies the assumptions of Chernoff's theorem with $-H= -(\frac{1}{2}\nabla^*\nabla + V)$ as infinitesimal generator, and on the other hand the product $P_{\Delta_1 \tau} \cdots P_{\Delta_N\tau} u$ is equal to the path integral over $H_{x;\tau}(M)$ appearing in Thm.~\ref{ThmPathIntegralManifold}, for any partition $\tau = \{0 = \tau_0 < \tau_1 < \dots < \tau_N = t\}$. B\"ar and Pf\"affle  do this using the short-time asymptotic expansion of the heat kernel; another route is to to rewrite the right hand side of \eqref{ProperFamily} as an integral over $T_x M$ and then do an appropriate time-rescaling which brings into play the geodesic flow on the tangent bundle (see \cite{LudewigBoundary}).

The advantage of the latter approach is that it readily generalizes to the case that $M$ is a compact manifold with boundary. In this case, one replaces the spaces $H_{x;\tau}(M)$ of piecewise geodesics with the spaces $H^{\mathrm{refl}}_{x;\tau}(M)$ of {\em reflected} piecewise geodesics. Here a reflected geodesic is a path $\gamma: [a, b] \rightarrow M$ that is a geodesic except at finitely many points $s \in [a, b]$, where $\gamma$ reflects at the boundary with angle of incidence equal to angle of reflection. One can treat a variety of different boundary conditions using this approach, including Dirichlet- and Neumann boundary conditions (however, Robin boundary conditions do not belong to this class). For a general presentation, we refer to \cite{LudewigBoundary} or \cite[Section~2]{LudewigThesis}; here we only give the results in the scalar case.

\begin{theorem}[Manifolds with Boundary]
Let $M$ be a compact Riemannian manifold with boundary, $V \in C^\infty(M)$ and let $u(t, x)$ be a solution to the heat equation with potential \eqref{HeatEquationWithPotential} with Dirichlet or Neumann boundary conditions. Then we have
\begin{equation*}
  u(t, x) = \lim_{|\tau|\rightarrow 0} \fint_{H^{\mathrm{refl}}_{x;\tau}(M)} \exp\bigl(-S[\gamma]\bigr) u_0\bigl(\gamma(t)\bigr) \dd^{\Sigma\text{-}H^1} \gamma
\end{equation*}
in the Neumann case and
\begin{equation*}
  u(t, x) = \lim_{|\tau|\rightarrow 0} \fint_{H^{\mathrm{refl}}_{x;\tau}(M)} \exp\bigl(-S[\gamma]\bigr) u_0\bigl(\gamma(t)\bigr) \,(-1)^{\mathrm{refl}(\gamma)} \dd^{\Sigma\text{-}H^1} \gamma
\end{equation*}
in the Dirichlet case. Here the action $S$ is given by \eqref{ActionWithPotential} and $\mathrm{refl}(\gamma)$ denotes the number of boundary reflections of the path $\gamma \in H^{\mathrm{refl}}_{x;\tau}(M)$. As before, these results hold in the case that $u_0$ is in any of the spaces $C^0(M)$ or $L^p(M)$, $1 \leq p < \infty$, with convergence in the respective space.
\end{theorem}

Let us finish this section with a quick discussion about what happens if one takes other metrics than the discrete $H^1$ metric \eqref{DiscreteH1} on $H_{x;\tau}(M)$. Initially, it may seem that the continuous version \eqref{ContinuousH1} is the more natural choice; however, there are results that indicate otherwise. Namely, Lim shows (under further curvature assumptions) that for a continuous function $f$ on $C([0, t], M)$, one has
\begin{equation*}
  \lim_{|\tau|\rightarrow 0} \fint_{H_{x;\tau}(M)} \exp\bigl(-S_0[\gamma]\bigr) f(\gamma) \dd^{H^1}\gamma = \int_{C([0, t], M)} f(\gamma) \det \bigl(\id + K_\gamma\bigr) \dd \mathbb{W}_x(\gamma),
\end{equation*}
where $\mathbb{W}_x$ is the (curved) Wiener measure for paths starting at $x$ and $K_\gamma$ is a certain trace-class integral operator on $L^2([0, t], \gamma^*TM)$ which is given in terms of the curvature along $\gamma$ (see Def.~1.13 and Thm.~1.14 in \cite{LimPathIntegrals}). In summary, path integral formulas become considerably more complicated in this case.

One can also consider the $L^2$ scalar product
\begin{equation} \label{L2Metric}
  (X, Y)_{L^2} = \int_0^1 \bigl\langle X(s), Y(s) \bigr\rangle \dd s
\end{equation}
or its discretized version
\begin{equation} \label{DiscreteL2}
  (X, Y)_{\Sigma\text{-}L^2} = \sum_{j=1}^N \bigl\langle X(\tau_j), Y(\tau_j)\bigr\rangle \Delta_j\tau
\end{equation}
on $H_{x;\tau}(M)$. In terms of the discrete $L^2$-metric, the solution $u(t, x)$ to the heat equation with potential \eqref{HeatEquationWithPotential} is given by
\begin{equation} \label{L2Approximation}
  u(t, x) = \lim_{|\tau|\rightarrow 0} \prod_{j=1}^N (\Delta_j\tau)^{-n/2} \fint_{H_{x;\tau}(M)} \exp\left(-S[\gamma] - \frac{1}{6}\int_0^t \mathrm{scal}\bigl(\gamma(s)\bigr) \dd s \right) u_0\bigl(\gamma(t)\bigr) \dd^{\Sigma\text{-}L^2}\gamma.
\end{equation}
Hence taking a different Riemannian metric on the path space (hence another volume form) results in a different pre-factor (which in this case also depends on the partition $\tau$ and not only on the dimension of the path space) and also a scalar curvature term appears (this result can be found in \cite{AnderssonDriver} and \cite{bpapproximation}). In the case of the continuous $L^2$ metric, the pre-factor in front of the scalar curvature becomes $(3+2\sqrt{3})/60$ instead of $1/6$ and one gets an again different $\tau$-dependent normalization factor (see Thm.~2.1 in  \cite{Laetsch}).

Conceptually, the $L^2$ metrics are somewhat less natural compared to the $H^1$ metric, since the action functional $S_0$ is not defined on $L^2$. 

\section{Approximation of the Heat Kernel}

Let $M$ be a compact Riemannian manifold. The heat kernel of the operator $\frac{1}{2}\Delta$ of $M$, denoted by $K_t(x, y)$, is characterized by the property that it solves the heat equation \eqref{HeatEquation} in the $x$ variable with the initial condition $\delta_y$, the delta distribution at $y$. For $t>0$, it is smooth in both $x$ and $y$. Analogous to the above, one should be able to write the heat kernel as a path integral, this time over the spaces $H_{xy;\tau}(M)$ of piecewise geodesics that travel from $x$ to $y$.

\begin{definition}
For a partition $\tau = \{0 = \tau_0 < \tau_1 < \dots < \tau_N = t\}$ of the interval $[0, t]$, the space $H_{xy;\tau}(M)$ is the space of paths $\gamma \in H_{x;\tau}(M)$ with $\gamma(t) = y$ such that $\gamma|_{[\tau_{j-1}, \tau_j]}$ is the unique shortest geodesic between its end points for each $j=1, \dots, N$.
\end{definition}

Clearly, via the evaluation map $\mathrm{ev}_\tau: H_{xy;\tau}(M) \rightarrow M^{N-1}$, $\gamma \mapsto (\gamma(\tau_1), \dots, \gamma(\tau_{N-1}))$, maps $H_{xy;\tau}(M)$ injectively onto an open subset of $M^{N-1}$. This determines a manifold structure on $H_{xy;\tau}(M)$ (compare also  \cite[Lemma~16.1]{MilnorMorseTheory}). Let $H_{x;\tau}^\circ(M) \subseteq H_{x;\tau}(M)$ be the subset of paths $\gamma$ such that $\gamma|_{[\tau_{j-1}, \tau_j]}$ is the unique shortest geodesic between its end points for each $j=1, \dots, N$ (it turns out that Thm.~\ref{ThmPathIntegralManifold} also holds when one replaces $H_{x;\tau}(M)$ by $H_{x;\tau}^\circ(M)$, compare \cite[Thm.~5.1]{bpapproximation}). Let us equip the spaces $H_{xy;\tau}(M)$ with a measure $\dd^{\Sigma\text{-}H^1}\gamma$ that satisfies the co-area formula
\begin{equation} \label{CoAreaFormula}
  \int_{H_{x;\tau}^\circ(M)} f(\gamma)\, \dd^{\Sigma\text{-}H^1}\gamma = t^{-n/2} \int_M \int_{H_{xy;\tau}(M)} f(\gamma)\, \dd^{\Sigma\text{-}H^1}\gamma \,\dd y 
\end{equation}
for continuous bounded functions $f$ on $H_{x;\tau}^\circ(M)$. Such a measure is constructed in Section~2.2.1 of \cite{LudewigThesis} and it is explicitly given by (setting $x_0 := x$, $x_N := y$)
\begin{equation} \label{DiscreteVolume}
\begin{aligned}
  &\int_{H_{xy;\tau}(M)} f\bigl(\gamma(\tau_1), \dots, \gamma(\tau_{N-1})\bigr) \dd^{\Sigma\text{-}H^1} \gamma \\
  &~~~~~~:= t^{n/2}\int_{M^{N-1}} f(x_1, \dots, x_{N-1}) \left(\prod_{j=1}^N J(x_{j-1}, x_j)(\Delta_j\tau)^{n/2} \right)^{-1}\dd x_1 \cdots \dd x_{N-1},
\end{aligned}
\end{equation}
where $J(x, y)$ is the Jacobian of the exponential map, given by $J(x, y) = \det(g_{ij}(y))^{1/2}$, with $g_{ij}$ being the coefficients of the metric in Riemannian normal coordinates around $x$.
 For this measure, one then obtains from Thm.~\ref{ThmPathIntegralManifold} that
\begin{equation} \label{HeatKernelConvergence}
  (2\pi t)^{-n/2} \fint_{H_{xy;\tau}(M)} \exp\bigl(-S_0[\gamma]\bigr) \dd^{\Sigma\text{-}H^1} \gamma ~\longrightarrow~ K_t(x, y)
\end{equation}
as $|\tau|\rightarrow 0$, where the convergence holds in the weak-$*$-topology of the space of finite Borel measures on $M$ (with respect to one of the variables). Namely, if we define the approximate heat operator $P_\tau$ on $C^0(M)$ by the formula
\begin{equation*}
  (P_\tau u_0)(x) := \fint_{H_{x;\tau}^\circ(M)}  \exp\bigl(-S_0[\gamma]\bigr) u_0\bigl(\gamma(t)\bigr) \dd^{\Sigma\text{-}H^1}\gamma
\end{equation*} 
for a partition $\tau= \{0 = \tau_0 < \tau_1 < \dots < \tau_N = t\}$ and the approximate heat kernel by
\begin{equation*}
  K_\tau(x, y) := (2\pi t)^{-n/2} \fint_{H_{xy;\tau}(M)} \exp\bigl(-S_0[\gamma]\bigr) \dd^{\Sigma\text{-}H^1},
\end{equation*}
then  for each $x \in M$ and each $u_0 \in C^0(M)$, 
\begin{equation*}
  \bigl(K_\tau(x, -), u_0\bigr)_{L^2} \stackrel{ \eqref{CoAreaFormula} }{=} (\delta_x, P_\tau u_0)_{L^2} \stackrel{\text{Thm.~\ref{ThmPathIntegralManifold}}}{\longrightarrow} (\delta_x, e^{-t\Delta} u_0)_{L^2} = \bigl(K_t(x, -), u_0\bigr)_{L^2}.
\end{equation*}
Of course, this result is not at all satisfactory, and one would like the convergence \eqref{HeatKernelConvergence} to hold pointwise. A partial result can be found in \cite{bpapproximation}, who prove at least that there exists some sequence $\tau^{(k)}$ of partitions with $|\tau^{(k)}| \rightarrow 0$ for which \eqref{HeatKernelConvergence} holds pointwise uniformly for $(x, y) \in M\times M$. This result follows from  Thm.~\ref{ThmPathIntegralManifold} with a little more effort. However, even the following stronger result is true \cite[Thm.~2.2.7]{LudewigThesis}.

\begin{theorem}[Heat Kernel Approximation] \label{ThmHeatKernelPathIntegral}
  Let $M$ be a compact Riemannian manifold and let $K_t^H(x, y)$ be the heat kernel of the operator $H = \frac{1}{2} \Delta + V$. Then we have
  \begin{equation*}
    K_t^H(x, y) = \lim_{|\tau|\rightarrow 0} (2\pi t)^{-n/2}\fint_{H_{xy;\tau}(M)} \exp\bigl(-S[\gamma]\bigr) \dd^{\Sigma\text{-}H^1} \gamma,
  \end{equation*}
  where $S$ is the action defined in \eqref{ActionWithPotential} and the slash over the integral sign denotes division by $(2 \pi)^{\dim(H_{xy;\tau}(M))/2}$. The convergence is uniform over $M\times M$.
\end{theorem}

This is result was proved be the author using heat kernel approximation and actually implies Thm.~\ref{ThmPathIntegralManifold}: The latter follows from integrating over the $y$ variable and using the co-area formula \eqref{CoAreaFormula}. The method is inspired by the paper \cite{baerrenormalization}, where a similar result is given for the discretized $L^2$-metric \eqref{L2Metric} on $H_{xy;\tau}(M)$ (the restriction of which to $H_{xy;\tau}(M)$ trivially satisfies the corresponding co-area formula \eqref{CoAreaFormula} with $L^2$ replaced by $H^1$). One obtains that  $K_t^H(x, y)$ is also given by
\begin{equation*}
  \lim_{|\tau|\rightarrow 0}  (2\pi t)^{-n/2} \prod_{j=1}^N (\Delta_j\tau)^{-n/2} \fint_{H_{xy;\tau}(M)} \exp\left(-S[\gamma] - \frac{1}{6} \int_0^t \mathrm{scal}\bigl(\gamma(s)\bigr) \dd s \right) \dd^{\Sigma\text{-}L^2} \gamma,
\end{equation*}
analogous to \eqref{L2Approximation}. The idea is to use the Markhov property of the heat kernel,
\begin{equation} \label{MarkhovProperty}
  K_t^H(x, y) = \bigl(K_{\Delta_1\tau}^H * \cdots * K_{\Delta_N\tau}^H\bigr)(x, y)
\end{equation}
for any partition $\tau=\{0 = \tau_0 < \tau_1 < \dots < \tau_N = t\}$ of the interval $[0, t]$, and then to replace the kernels $K_{\Delta_j\tau}$ by suitable approximations to discover a path integral on the right hand side. Here, for two kernels $k(x, y)$ and $\ell(x, y)$, the convolution is defined by
\begin{equation*}
  (k * \ell)(x, y) = \int_M k(x, z) \ell(z, y) \dd z.
\end{equation*}
The mentioned approximation result is then the following (for reference, see \cite[Thm.~2.1.12]{LudewigThesis}; compare also Thm.~1.2 in \cite{LudewigStrongAsymptotics}).

\begin{theorem}[Convolution Approximation] \label{ThmConvolutionApproximation0}
Let $M$ be a compact Riemannian manifold and let $k_t$, $\ell_t \in L^\infty(M \times M)$ be two time-dependent kernels. For $T, R>0$ and $m \in \N$, suppose that there exist $\gamma_1, \gamma_2 \in \R$ and $c, \nu, \alpha_1, \dots, \alpha_m, \beta_1, \dots, \beta_m\geq0$ satisfying $\alpha_i + \beta_i/2 \geq 1 + \nu$ for each $1 \leq i \leq m$ such that for all $0 < t \leq T$ and all $x, y \in M$, we have
\begin{equation}\label{Estimate1}
 |\ell_t(x, y)|, |k_t(x, y)| \leq e^{\gamma_1 t + \gamma_2 d(x, y)^2} K_t(x, y),
\end{equation}
and for all $0 < t \leq T$ and all $x, y \in M$ with $d(x, y) < R$, we have
\begin{equation}\label{Estimate2}
  |k_t(x, y) - \ell_t(x, y)| \leq c \sum_{j=1}^m t^{\alpha_j} d(x, y)^{\beta_j} K_t(x, y).
\end{equation}
Then there exist constants $C, \delta>0$ such that for each partition $\tau = \{0 = \tau_0 < \tau_1 < \dots < \tau_N = t\}$ of intervals $[0, t]$ with $0 < t \leq T$ and $|\tau| \leq \delta t$, we have
\begin{equation*}
  \bigl|k_{\Delta_1\tau} * \cdots * k_{\Delta_N\tau} - \ell_{\Delta_1 \tau} * \cdots * \ell_{\Delta_N\tau}\bigr| \leq C t^{1-\beta/2}\, |\tau|^{\nu} K_t
\end{equation*}
uniformly on $M \times M$. Here, $\beta := \max_{1\leq i\leq m} \beta_i$ and $K_t$ is the heat kernel of the operator $\frac{1}{2}\Delta$ on $M$.
\end{theorem}

This can be used as follows to prove Thm.~\ref{ThmHeatKernelPathIntegral}. Set $k_t(x, y) := K^H_t(x, y)$ and set $\ell_t(x, y)$ equal to
\begin{equation*}
  (2 \pi t)^{-n/2} \exp\left(-S[\gamma] + \frac{1}{12}\int_0^t\mathrm{scal}\bigl(\gamma_{xy;t}(s)\bigr)\dd s - \frac{1}{12}\mathrm{ric}(\dot{\gamma}_{xy;t}(0), \dot{\gamma}_{xy;t}(0))\right) J(x, y)^{-1},
\end{equation*}
where $\gamma_{xy;t}$ is the unique shortest minimizing geodesic between $x$ and $y$ parametrized by $[0, t]$ (which is defined for almost all pairs $(x, y) \in M \times M$). Using the short-time asymptotic expansion of the heat kernel (see e.g.\ \cite{bgv}, \cite{gilkey95}, \cite{LudewigStrongAsymptotics}), one then verifies that $k_t$ and $\ell_t$ satisfy the assumptions of Thm.~\ref{ThmConvolutionApproximation0} with $m=3$, $\alpha_1 = 2$, $\beta_1 = 0$, $\alpha_2 = 1$, $\beta_2 = 1$, $\alpha_3 = 0$, $\beta_3 = 0$ and $\nu = 1/2$. Hence there exists $C>0$ such that
\begin{equation*}
  \bigl| K_t^H - \ell_{\Delta_1\tau} * \cdots * \ell_{\Delta_N\tau}\bigr| \leq C \left(\frac{|\tau|}{t}\right)^{1/2} p_t
\end{equation*}
for all $t$ small enough and all partitions of the interval $[0, t]$ fine enough. Using the definition \eqref{DiscreteVolume} of the discrete volume on $H_{xy;\tau}(M)$, we then obtain
\begin{equation} \label{ConvolutionPathIntegral}
  \bigl(\ell_{\Delta_1\tau} * \cdots * \ell_{\Delta_N\tau}\bigr)(x, y) = (2 \pi t)^{-n/2} \fint_{H_{xy;\tau}(M)} \exp\bigl(-S[\gamma]\bigr) \, F_\tau (\gamma)\,  \dd^{\Sigma\text{-}H^1}\gamma,
\end{equation} 
where $F_\tau$ is the function
\begin{equation*}
  F_\tau (\gamma) := \exp\left( \frac{1}{12}\int_0^t\mathrm{scal}\bigl(\gamma(s)\bigr)\dd s - \frac{1}{12} \sum_{j=1}^N \mathrm{ric}(\Delta_j\gamma, \Delta_j\gamma)\right).
\end{equation*}
Here $\Delta_j\gamma := \dot{\gamma}(\tau_{j-1}+)\Delta_j\tau$, $j=1, \dots, N$, are the "increments" of $\gamma$. One now discovers that the scalar curvature integral is the quadratic variation of Brownian motion, while the right hand side is the discrete approximation of this. We therefore have
\begin{equation*}
  \lim_{|\tau|\rightarrow 0} \sum_{j=1}^N \mathrm{ric}(\Delta_j\gamma, \Delta_j\gamma) = \int_0^t\mathrm{scal}\bigl(\gamma(s)\bigr)\dd s,
\end{equation*} 
with the convergence being in measure with respect to the Wiener measure (compare Prop.~3.23 in \cite{EmeryStochastic}). Using this, one can show that one can drop the term $F_\tau$ in the integrand of \eqref{ConvolutionPathIntegral} and still end up with the same result in the limit $|\tau|\rightarrow 0$.

\section{Short-Time Expansions of Path Integrals}

In our previous discussions, the action functional $S_0$ and the $H^1$ metric played a prominent role. Even though it didn't appear explicitly in the above, we always had the space
\begin{equation*}
  H_{xy}(M) := \bigl\{ \gamma \in C([0, 1], M) \mid \gamma(0) = x, \gamma(1) = y, S_0[\gamma] < \infty \bigr\}
\end{equation*}
of finite action paths lurking in the background. This is an infinite-dimensional manifold modelled on a Hilbert space, and the continuous $H^1$ metric \eqref{ContinuousH1} turns it into an infinite-dimensional Riemannian manifold. It is equal to the space of paths which have Sobolev regularity $H^1$ in local charts and by definition the largest space where the action functional $S_0$ is defined. In this section, we will discuss the relations between the heat kernel representations of the path integrals discussed above und this infinite-dimensional path space.

By rescaling the paths, we can always arrange the domain of definition of our paths to be the interval $[0, 1]$, which we will be doing henceforth. This has the advantage that the dependence on the $t$ parameter becomes more clear in the formulas. Looking at Thm.~\ref{ThmHeatKernelPathIntegral}, we may get the idea to write the heat kernel $K_t(x, y)$ of the operator $\frac{1}{2}\Delta$ as an integral over the infinite-dimensional manifold $H_{xy}(M)$, 
\begin{equation} \label{FormalPathIntegral}
  K_t(x, y) \stackrel{\text{formally}}{=} (2 \pi t)^{-n/2} \fint_{H_{xy}(M)} \exp \bigl(-S_0[\gamma]/ t\bigr) \,\dd^{H^1}\gamma,
\end{equation}
where the $1/t$ in the exponent comes from the rescaling of the domain of definition of the paths. Of course, this is only formal, as can be seen e.g.\ by the fact that the normalization constant indicated by the slash over the integral sign should now be $(2\pi t)^{\dim(H_{xy}(M))}$, which does not make sense. However, taking the formula \eqref{FormalPathIntegral} seriously for the moment, we notice that the integral has the form of a Laplace integral, which can be evaluated using the method of stationary phase as $t \rightarrow 0$ (see e.g.\ \cite[Section~1.2]{Duistermaat}): Suppose that $x$ and $y$ are close to each other, so that there is a unique minimizing geodesic $\gamma_{xy} \in H_{xy}(M)$ connecting the two. Then this is the unique non-degenerate minimum of the action functional on $H_{xy}(M)$, and we have $S_0[\gamma_{xy}] = d(x, y)^2/2$. The method of stationary phase therefore asserts that
\begin{equation} \label{FormalStationaryPhase}
 (2 \pi t)^{-n/2} \fint_{H_{xy}(M)} \exp \bigl(-S_0[\gamma]/ t\bigr) \,\dd^{H^1}\gamma ~\sim~ (2 \pi t)^{-n/2} e^{-\frac{d(x, y)^2}{2t}}\,\det\bigl(\nabla^2 S_0|_{\gamma_{xy}}\bigr)^{-1/2},
\end{equation}
where the asymptotic relation means that the quotient of the two sides converges to one as $t \rightarrow 0$. Here the Hessian of the action $S_0$ comes into play, which is well-known to essentially be given by the Jacobi differential equation, 
\begin{equation*}
\begin{aligned}
  \nabla^2 S_0|_\gamma[X, Y] &= \int_0^1 \Bigl( \bigl\langle \nabla_s X(s),  \nabla_s Y(s) \bigr\rangle + \bigl\langle R\bigl(\dot{\gamma}(s), X(s)\bigr) \dot{\gamma}(s), Y(s)\bigr\rangle \Bigr)\dd s\\
  &= (X, Y)_{H^1} + (\mathcal{R}_\gamma X, Y)_{L^2},
\end{aligned}
\end{equation*}
where we set $\mathcal{R}_\gamma(s) := R(\dot{\gamma}(s), -) \dot{\gamma}(s)$ for the {\em Jacobi endomorphism} (see e.g.\ \cite[Thm.~13.1]{MilnorMorseTheory}). Dualizing with respect to the $H^1$ metric, the Hessian $\nabla^2 S_0|_\gamma$ is therefore given by the operator $\mathrm{id} + (-\nabla_s^2)^{-1} \mathcal{R}_\gamma$. The "perturbation" $(-\nabla_s^2)^{-1} \mathcal{R}_\gamma$ is easily seen to be trace-class so that the $H^1$-determinant appearing in \eqref{FormalStationaryPhase} is well-defined as a Hilbert space determinant. It is therefore natural to ask if \eqref{FormalStationaryPhase} indeed gives the correct heat kernel asymptotics. This turns out to be the case, even in the degenerate case, when $x, y$ are in each other's cut locus.

\begin{theorem}[Heat Kernel Asymptotics] \label{ThmFirstOrderAsymptotics}
  Let $M$ be an $n$-dimensional compact Riemannian manifold with heat kernel $K_t(x, y)$. For $x, y \in M$, suppose that the set $\Gamma_{xy}^{\min}$ of minimizing geodesics between $x$ and $y$ is a $k$-dimensional submanifold of $H_{xy}(M)$ which is non-degenerate in the sense that for $\gamma \in \Gamma_{xy}^{\min}$, the Hessian $\nabla^2S_0|_\gamma$ is non-degenerate when restricted to the normal space $N_\gamma \Gamma_{xy}^{\min}$ of $\Gamma_{xy}^{\min}$ in $H_{xy}(M)$. Then we have
  \begin{equation*}
    \lim_{|\tau|\rightarrow 0} (2 \pi t)^{n/2+k/2} e^{\frac{d(x, y)^2}{2t}} K_t(x, y) = \int_{\Gamma_{xy}^{\min}} \det\bigl(\nabla^2 S_0|_{N_\gamma \Gamma_{xy}^{\min}}\bigr)^{-1/2} \dd^{H^1} \gamma,
  \end{equation*}
  where we integrate with respect to the Riemannian volume measure on $\Gamma_{xy}^{\min}$ determined by the continuous $H^1$ metric.
\end{theorem}

These are exactly the asymptotics that are to be expected when one formally takes the stationary phase expansion of the path integral \eqref{FormalPathIntegral}, also in the degenerate case when $\dim(\Gamma_{xy}^{\min}) > 0$ (compare the appendix in \cite{LudewigStrongAsymptotics}). Of course, if $x$ and $y$ are close, then $\Gamma_{xy}^{\min} = \{\gamma_{xy}\}$ and the integral is just the evaluation at $\gamma_{xy}$, so that Thm.~\ref{ThmFirstOrderAsymptotics} reduces to the result formally obtained in \eqref{FormalStationaryPhase}. In this case, it is furthermore known that the heat kernel has the short-time asymptotic expansion
\begin{equation*}
  K_t(x, y) ~\sim~ (2 \pi t)^{-n/2} e^{-\frac{d(x, y)^2}{2t}} J(x, y)^{-1/2}
\end{equation*}
involving the Jacobian of the Riemannian exponential map $J(x, y)$ that already appeared in \eqref{DiscreteVolume}. Comparing with Thm.~\ref{ThmFirstOrderAsymptotics}, one therefore obtains the following corollary.

\begin{corollary}
The Jacobian of the exponential map $J(x, y) = \det( d\exp_x|_{\dot{\gamma}_{xy}(0)}) = \det(g_{ij}(y))^{1/2}$ is equal to the Hilbert space determinant of the action functional on the tangent space of $H_{xy}(M)$ at $\gamma_{xy}$,
\begin{equation*}
  J(x, y) = \det\bigl(\nabla^2 S_0|_{\gamma_{xy}}\bigr).
\end{equation*}
where the determinant is taken with respect to the continuous $H^1$ metric \eqref{ContinuousH1}.
\end{corollary}

At first glance, one could think that Thm.~\ref{ThmFirstOrderAsymptotics} follows by evaluating the finite-dimensional path integrals over $H_{xy;\tau}(M)$ in Thm.~\ref{ThmHeatKernelPathIntegral} with the method of stationary phase and then taking the limit $|\tau|\rightarrow 0$. A problem is however that Thm.~\ref{ThmHeatKernelPathIntegral} gives no control over the time-uniformity of the approximation, i.e.\ there is no reason one should be allowed to exchange taking the mesh-limit $|\tau|\rightarrow 0$ and the time-limit $t\rightarrow 0$. This can be fixed though by using Thm.~\ref{ThmConvolutionApproximation0}, which provides careful error estimates of the approximation. More concretely, one proves a different theorem on approximating the heat kernel by integrals over $H_{xy;\tau}(M)$ which is time-uniform; the downside is that one picks up a more complicated integrand involving curvature terms and which then disappears again in the short-time limit. For detais, refer to \cite{LudewigThesis} or \cite{LudewigDeterminants}.

\medskip

In applications from physics, path integrals are often regularized using zeta determinants, see e.g.\ \cite{Hawking} or \cite{WittenDiracOperators}. In our case, instead of taking the determinant of $\nabla^2 S_0|_\gamma$ with respect to the $H^1$ metric, i.e.\ the Hilbert space determinant of the operator $\mathrm{id} + (\nabla_s^2)^{-1} \mathcal{R}_\gamma$, one can also take the zeta-regularized determinant of the operator $\nabla_s^2 + \mathcal{R}_\gamma$ on $L^2([0, 1], \gamma^*TM)$ (with Dirichlet boundary conditions, since our paths have fixed endpoints). This leads to the following result.

\begin{theorem}[Heat Kernel Asymptotics II] \label{ThmFirstOrderAsymptoticsII}
Under the assumptions of Thm.~\ref{ThmFirstOrderAsymptotics}, we equivalently have
  \begin{equation*}
    \lim_{|\tau|\rightarrow 0} (2 \pi t)^{n/2+k/2} e^{\frac{d(x, y)^2}{2t}} p_t(x, y) = \int_{\Gamma_{xy}^{\min}} \frac{\det\nolimits_\zeta(-\nabla_s^2)^{1/2}}{\det\nolimits_\zeta^\prime\bigl(-\nabla_s^2 + \mathcal{R}_\gamma\bigr)^{1/2}} \dd^{H^1} \gamma,
  \end{equation*}
  where we integrate with respect to the Riemannian volume measure on $\Gamma_{xy}^{\min}$ with respect to the continuous $H^1$ metric.
\end{theorem}

For a proof and further discussion of in Thm.~\ref{ThmFirstOrderAsymptoticsII}, we again refer to \cite[Section~3.2]{LudewigThesis} or \cite{LudewigDeterminants}.

\begin{remark}Here, for an unbounded self-adjoint linear operator $L$ on a Hilbert space with non-negative eigenvalues, the zeta-function is defined by
\begin{equation*} 
  \zeta_L(z) = \sum_{\lambda >0} \lambda^{-z},
\end{equation*}
where the sum goes over the non-negative eigenvalues of $L$; it converges for $\mathrm{Re}(z)$ large if the eigenvalues grow sufficiently fast. For elliptic operators (with suitable boundary conditions), one can show that $\zeta_L$ of has a meromorphic extension to the complex plane which is regular at zero. The zeta-regularized determinant of $L$ is then defined by $\det\nolimits_\zeta(L) = \exp(-\zeta^\prime(0))$ if $L$ has only positive eigenvalues and  $\det\nolimits_\zeta(L) = 0$ otherwise. In the case that $L$ has zero eigenvalues, we furthermore define $\det\nolimits_\zeta^\prime(L) = \exp(-\zeta^\prime(0))$. Notice that $\det\nolimits_\zeta^\prime\bigl(-\nabla_s^2 + \mathcal{R}_\gamma\bigr)$ in Thm.~\ref{ThmFirstOrderAsymptoticsII} is the regularized analog of $\det(\nabla^2 S_0|_{N_\gamma \Gamma_{xy}^{\min}})$ in Thm.~\ref{ThmFirstOrderAsymptotics}.
\end{remark}

\begin{remark}
The zeta determinant  of the operator $-\nabla_s^2$ on $L^2([0, 1], \gamma^*TM)$ with Dirichlet boundary conditions can easily calculated to be $\det\nolimits_\zeta(-\nabla_s^2) = 2^n$.
\end{remark}

    \bibliography{Literatur}

\end{document}